\documentclass[13pt]{amsart}
\usepackage{amssymb,amsmath,amsthm,mathrsfs,verbatim,url, graphicx}
\usepackage{xcolor}
\newcommand{\sm}[4]{\left(\begin{smallmatrix}#1&#2\\ #3&#4 \end{smallmatrix}
\right)}

\newtheorem{theorem}{Theorem}
\newtheorem{lemma}[theorem]{Lemma}
\newtheorem{corollary}[theorem]{Corollary}
\newtheorem{conjecture}{Conjecture}
\newtheorem{proposition}[theorem]{Proposition}
\newtheorem{definition}[theorem]{Definition}

\theoremstyle{remark}

\newtheorem*{remark}{Remark}

\newtheorem*{questions}{Questions}
\newtheorem*{example}{Example}
\numberwithin{theorem}{section} \numberwithin{equation}{section}

\newcommand{\Tr}{{\text {\rm Tr}}}
\newcommand{\Q}{\mathbb{Q}}

\newcommand{\Z}{\mathbb{Z}}
\newcommand{\N}{\mathbb{N}}
\newcommand{\SL}{{\text {\rm SL}}}
\newcommand{\GL}{{\text {\rm GL}}}

\newcommand{\F}{\mathbb{F}_p}

\newcommand{\Gal}{\rm Gal}
\newcommand{\X}{\mathcal{X}}
\newcommand{\Nr}{\rm Nr}

\author{Matija Kazalicki}
\address{Department of Mathematics, University of Zagreb, Bijeni\v cka cesta 30, 10000 Zagreb, Croatia}
\email{matija.kazalicki@math.hr}
\thanks{MK acknowledges support from the QuantiXLie Center of
Excellence}
\author{Daniel Kohen}
\address{Departamento de Matem\'atica, Facultad de Ciencias Exactas y Naturales, Universidad de Buenos Aires and IMAS, CONICET, Argentina}
\email{dkohen@dm.uba.ar}
\thanks{DK was partially supported by a CONICET doctoral fellowship}
\title{Supersingular zeros of divisor polynomials of elliptic curves of prime conductor}
\begin{document}

\maketitle

\begin{abstract}
For a prime number $p$ we study the zeros modulo $p$ of divisor polynomials of 
rational elliptic curves $E$ of conductor $p$. Ono \cite[p.~118]{Ono} made 
the 
observation that these zeros of are often $j$-invariants of supersingular 
elliptic curves over $\overline{\F}$. We show that these supersingular zeros are 
in bijection with zeros modulo $p$ of an associated quaternionic modular form 
$v_E$. 

This allows us to prove that if the root number of $E$ is $-1$ then all 
supersingular $j$-invariants of elliptic curves defined over $\mathbb{F}_{p}$ 
are zeros of the corresponding divisor polynomial. 

If the root number is $1$ we study the discrepancy between rank $0$ and higher 
rank elliptic curves, as in the latter case the amount of supersingular zeros 
in $\mathbb{F}_p$ seems to be larger.
In order to partially explain this phenomenon, we conjecture that when $E$ has 
positive rank the values of the coefficients of $v_E$ corresponding to 
supersingular elliptic curves defined over $\mathbb{F}_p$ are even. We prove 
this conjecture in the case when the discriminant of $E$ is positive, and 
obtain several other results that are of independent interest.

\end{abstract}

\section{Introduction}
Let $E$ be a rational elliptic curve of prime conductor $p$. Denote by $f_E(\tau)\in S_2(\Gamma_0(p))$ the newform associated to $E$ by the Shimura-Taniyama correspondence. 
Serre \cite[Theorem ~11]{Ser2} showed that there is an isomorphism between 
modular forms modulo $p$ of weight $p+1$ and level $1$ and modular forms modulo 
$p$ of weight $2$ and 
level $p$. More precisely he proved that $f_E(\tau) \equiv F_E(\tau) 
\pmod{p}$, where
\[
F_E(\tau)=\textrm{Trace}^{\Gamma_0(p)}_{\SL_2(\Z)}\left( f_E(\tau)\cdot( E_{p-1}(\tau)-p^{p-1}E_{p-1}(p\tau))\right)\in S_{p+1}(\SL_2(\Z)),
\]
and $E_{p-1}(\tau)$ is the normalized Eisenstein series of weight $p-1$.

Given $k \in \mathbb{Z}$ define

\[ \tilde{E}_{k}(\tau)= \left\{ \begin{array}{lcc}
             1 &   if  & k \equiv 0 \bmod{12} ,\\
             \\  E_{4}(\tau)^2 E_{6}(\tau) &   if  & k \equiv 2 \bmod{12}, \\
             \\  E_{4}(\tau) &   if  & k \equiv 4 \bmod{12}, \\
             \\  E_{6}(\tau) &   if  & k \equiv 6 \bmod{12}, \\
             \\  E_{4}(\tau)^2 &   if  & k \equiv 8 \bmod{12} ,\\
             \\   E_{4}(\tau)E_{6}(\tau)  &   if  & k \equiv 10 \bmod{12} ,
             \end{array}
   \right.  \]
where $E_{4}(\tau)$ and $E_{6}(\tau)$ are the classical Eisenstein series of weight $4$ and $6$ respectively.

Moreover, consider
\[ m(k)= \left\{ \begin{array}{lcc}
             \left \lfloor{\frac{k}{12}}\right \rfloor  &   if  &  k \not\equiv 
2 \bmod{12}, \\
             \\ \left \lfloor{\frac{k}{12}}\right \rfloor -1 &  if  &  k \equiv 
2 \bmod{12}.
             \end{array}
   \right.  \]

Given any $g \in M_{k}(SL_{2}(\Z))$ we obtain a rational function $\tilde{F}(g,x)$ which is characterized by the formula 

\[ \frac{g(\tau)}{\Delta(\tau)^{m(k)} \tilde{E}_{k}(\tau)}  = \tilde{F}(g,j(\tau))  ,\]

where $\Delta$ is the only weight $12$ and level $1$ cuspform and $j$ is the classical $j$-invariant. In order to define a polynomial and not just a rational function, define

\[ h_{k}(x)= \left\{ \begin{array}{lcc}
             1 &   if  & k \equiv 0 \bmod{12} ,\\
             \\  x^2 (x-1728) &   if  & k \equiv 2 \bmod{12}, \\
             \\  x &   if  & k \equiv 4 \bmod{12}, \\
             \\  x-1728 &   if  & k \equiv 6 \bmod{12}, \\
             \\  x^2 &   if  & k \equiv 8 \bmod{12} ,\\
             \\   x(x-1728)  &   if  & k \equiv 10 \bmod{12} .
             \end{array}
   \right.  \]

The \emph{divisor polynomial} is 

\[ F(g,x)= h_{k}(x)\tilde{F}(g,x) .\]

Ono \cite[p.~118]{Ono} made the observation that the zeros of $F(F_E,x) \bmod{p} 
\in \F[x]$ (in $\overline{\F}$) are often supersingular $j$-invariants (i.e. 
$j$-invariants of supersingular elliptic curves over $\overline{\F}$), and 
asked for an explanation for this. 

For example, if $E_{83}$ is the elliptic curve of conductor $83$ given by 
\[
E_{83}: \quad y^2+xy+y=x^3+x^2+x,
\]
then 
\begin{align*}
F_{E_{83}}(\tau) &\equiv  \Delta(\tau)E_4(\tau)^{18}+19 \Delta(\tau)^2 E_4(\tau)^{15}+21 \Delta(\tau)^3 E_4(\tau)^{12}\\
&+58 \Delta(\tau)^4 E_4(\tau)^9 + 21 \Delta(\tau)^5 E_4(\tau)^6+60\Delta(\tau)^6 E_4(\tau)^3 \pmod{83}.\\
\end{align*}
Since $j(\tau)=E_4(\tau)^3/\Delta(\tau)$, it follows that
$$F(F_{E_{83}},x)\equiv x(x+15)(x+16)(x+33)(x+55)(x+66) \pmod{83}.$$

In this case, the roots of $F(F_{E_{83}},x)$ in $\overline{\mathbb{F}}_{83}$ are precisely the supersingular $j$-invariants that lie in ${\mathbb{F}}_{83}$.

It is worth noting that the root number of $E_{83}$ is $-1$. The behavior of the roots of the divisor polynomial is explained by the following theorem. 

\begin{theorem}\label{thm:1}
Let $E/\Q$ be an elliptic curve of prime conductor $p$ with root number $-1$, 
and let $F(F_E,x)$ be the corresponding divisor polynomial. If $j\in \F$ is a 
supersingular $j$-invariant mod $p$, then $F(F_E,j)\equiv 0 \pmod{p}$.
\end{theorem}
 
If the root number of $E$ is $1$, the supersingular zeros of divisor polynomials are harder to understand.
Denote by $s_p$ the number of isomorphism classes of supersingular elliptic curves defined over $\F$. Eichler proved that
\[
s_p=
\begin{cases}
\frac{1}{2}h(-p) &\textrm{if } p \equiv 1 \pmod{4},\\
2h(-p) &\textrm{if }p \equiv 3 \pmod{8},\\
h(-p) &\textrm{if }p \equiv 7 \pmod{8},
\end{cases}
\]
where $h(-p)$ is the class number of the imaginary quadratic field $\Q(\sqrt{-p})$. See \cite{Gross} for an excellent exposition of Eichler's work.

Denote by $N_p(E)$ the number of $\F$-supersingular zeros of the divisor polynomial $F(F_E,x)$, i.e. $$N_p=\#\{j:j \in \F, F(F_E,j) \equiv 0 \bmod{p}  \textrm{ and $j$ is supersingular $j$-invariant} \}.$$

Figure \ref{fig:1} shows the graph of the function $\frac{N_p(E)}{s_p}$ where $E$ ranges over all elliptic curves of root number $1$ and conductor $p$ where $p<10000$. The elliptic curves of rank zero ($158$ of them) are colored in blue, while the elliptic curves of rank two ($59$ of them) are colored in red.
\begin{figure}[htp]\label{fig:1}
\centering
\includegraphics[scale = 0.33]{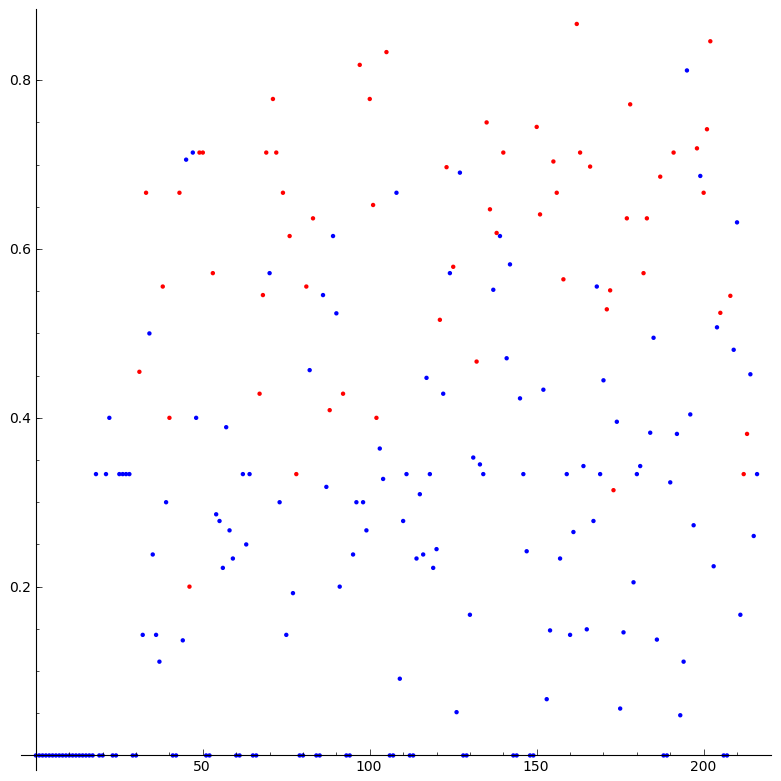}
\caption{$\frac{N_p}{s_p}$ for $p<10000$.}
\end{figure}

It would be interesting to understand this data. In particular, 
\begin{questions}\
\begin{enumerate}
	\item Why are there "so many" $\F$-supersingular zeros?
	\item How can we explain the difference between rank $0$ and rank $2$ curves?
	\item What about the outlying rank $0$ curves (e.g. of conductor $p=4283$ and $p=5303$) with the "large" number of zeros?
\end{enumerate}
\end{questions}

\begin{remark}
It seems that there is no obvious connection between the number of $\mathbb{F}_{p^2}$-supersingular zeros of the divisor polynomial $F(F_E,x)$ and the rank of elliptic curve $E$.    
\end{remark}

The key idea to study these questions is to show (following \cite{Serre2}) how to associate to $F_E$ a modular form $v_E$ on the quaternion algebra $B$ over $\Q$ ramified at $p$ and $\infty$. Such modular form is a function on the (finite) set of isomorphism classes of supersingular elliptic curves over $\overline{\F}$.  In order to explain this precisely we combine the expositions from \cite{Emerton} and \cite{Gross}.

Let $X_0(p)$ be the curve over $\mathrm{Spec}\, \Z$ that is a coarse moduli space for the $\Gamma_0(p)$-moduli problem. The geometric fiber of $X_0(p)$ in characteristic $p$ is the union of two rational curves meeting at $n=g+1$ ordinary double points: $e_1$, $e_2$, \ldots, $e_{n}$ ($g$ is the arithmetic genus of the fibers of $X_0(p)$.) They are in bijective correspondence with the isomorphism classes of supersingular elliptic curves $E_i/\overline{\F}$. Denote by $\mathcal{X}$ the free $\Z$-module of divisors supported on the $e_i$. The action of Hecke correspondences on the set of $e_i$ induces an action on $\X$. Explicitly, the action of the correspondence $t_m$ ($m\ge 1$) is given by the transpose  of the Brandt matrix $B(m)$
\[
t_m e_i = \sum_{j=1}^n B_{ij}(m)e_j.
\]
There is a correspondence between newforms of level $p$ and weight $2$ and modular forms for the quaternion algebra $B$ that preserves the action
of the Hecke operators.
Let $v_E=\sum v_E(e_i)e_i\in \X$  be an eigenvector for all $t_m$ corresponding to $f_E$, i.e. $t_m v_E= a(m) v_E$, where $f_E(\tau)=\sum_{m=1}^\infty a(m)q^m$. We normalize $v_E$ (up to the sign) such that the greatest common divisor of all its entries is $1$. We are now able to state the following
crucial theorem.
\begin{theorem} \label{thm:zeros}
Let $j=j(E_i)$ be the $j$-invariant of the supersingular elliptic curve $E_i$. Then 
$$F(F_E,j) \equiv 0 \bmod{p} \iff v_E(e_i) \equiv 0 \bmod{p}.$$ 
\end{theorem}

This theorem allows us to give a more explicit description of the supersingular zeros of the divisor polynomial. Furthermore it enables us to obtain computational data in a much more efficient manner.
The proof of Theorems \ref{thm:1} and \ref{thm:zeros} will be the main goal of Section \ref{section:proofs}. In order to prove them we will use
 both Serre's and Katz's theory of modular forms modulo $p$ and the modular forms introduced in \cite{Serre2}.

Now, let $D_E$ be the congruence number of $f_E$, i.e. the largest integer such 
that there exists a weight two cusp for on $\Gamma_0(p)$, with integral 
coefficients, which is orthogonal to $f_E$ with respect to the Petersson inner 
product and congruent to $f_E$ modulo $D_E$.
The congruence number is closely related to $\deg \phi_{f_E}$, the modular 
degree of $f_E$, which is the degree of the minimal parametrization 
$\phi_{f_E}:X_0(p)\rightarrow E'$ of the strong Weil elliptic curve $E'/\Q$ 
associated to $f_E$ ($E'$ is isogenous to $E$ but they may not be equal). In 
general, $\deg \phi_{f_E}|D_E$, and if the 
conductor of $E$ is prime, we have that $\deg \phi_{f_E}=D_E$ (see \cite{ARS}). 

The idea is to relate these concepts to the aforementioned quaternion modular 
form $v_E$.
Denote by $w_i=\frac{1}{2}\#\textrm{Aut}(E_i)$. It is known that $w=\prod_i w_i$ 
is equal to the denominator of $\frac{p-1}{12}$ and $\sum_{i=1}^n 
\frac{1}{w_i}=\frac{p-1}{12}$ (Eichler's mass formula). 
 We define a $\Z$-bilinear pairing 
\[
\langle -,- \rangle:\X\times\X \rightarrow \Z,
\]
by requiring $\langle e_i,e_j \rangle=w_i \delta_{i,j}$ for all $i,j \in \{1,\ldots,n \}$.

We have the following theorem due to Mestre  \cite[Theorem ~3]{Mestre}
\begin{theorem} \label{prop:modular} Using the notation above, we have
$$\langle v_E,v_E \rangle=t D_E,$$
where $t$ is the size of $E(\Q)_{tors}$.
\end{theorem}

We observe that the modular degree of the elliptic curves under consideration 
(of rank $0$ or $2$, conductor $p$, where $p<10000$) is ``small'', which 
suggests that the integral vector $v_E$ will have many zero entries. This gives 
a partial answer to Question 1. Zagier \cite[Theorem ~5]{Zagier} proved that if we consider
elliptic curves with bounded $j$-invariants we have 
\[\deg \phi_{f_E}<<p^{7/6}\log(p)^3. \]
 On the other 
hand, Watkins \cite[Theorem ~5.1]{Wat} showed that 
\[\deg \phi_{f_E} >> p^{7/6}/\log(p). \]

To address Questions $2$ and $3$ we focus on the mod $2$ behavior of $v_E$. Based on the numerical evidence we pose the following conjecture.

\begin{conjecture}\label{conj:1} If $E$ is an elliptic curve of prime conductor 
$p$, root number $1$, and $rank(E)>0$, then $v_E(e_i)$ is an even number for 
all $e_i$  with $j(E_i)\in \F$ 
\end{conjecture}

While this is true for all $59$ rank $2$ curves we observed, it holds for $35$
out of $158$ rank $0$ curves. This explains in a way a difference in the number 
of $\F$-supersingular zeros between rank $0$ and rank $2$ curves (Question $2$), 
since, heuristically, it seems more likely for a number to be zero if we know 
it is even (especially in light of Theorem \ref{prop:modular} which suggests 
that the numbers $v_E(e_i)$ are small.)

The thirty two out of thirty five elliptic curves of rank $0$ for which the 
conclusion of Conjecture \ref{conj:1} holds (the remaining three curves have 
conductors $p=571, 6451$ and  $8747$) are distinguished from the other rank $0$ 
curves by the fact that their set of real points $E(\mathbb{R})$ is not 
connected (i.e. $E$ has positive discriminant). In general, we have the 
following theorem, which will be the subject of Section \ref{section:mod2}.

\begin{theorem} \label{thm:even}Let $E/\Q$ be an elliptic curve of prime conductor $p$ such that 
\begin{enumerate}
	\item $E$ has positive discriminant
	\item $E$ has no rational point of order $2$,
\end{enumerate}
then $v_E(e_i)$ is an even number for all $e_i$  with $j(E_i)\in \F$.
\end{theorem}

Note that this gives a partial answer to Question 3 since, for example, all outlying elliptic curves of rank $0$ for which $\frac{N_p}{s_p}>0.5$ have positive discriminant and no rational point of order $2$.

Note that among 59 rank 2 curves, for 25 of them $E(\mathbb{R})$ is not 
connected (and have no rational point of order $2$). For the rest of the rank 2 
elliptic curves, we don't have an explanation of why they satisfy the 
conjecture.

Lastly, in the final section we will show how the Gross-Waldspurger formula 
might answer question $2$. More precisely, we will show that the quaternion 
modular form $v_E$ associated to an elliptic curve $E$ of rank $2$ must be 
orthogonal to divisors arising from optimal embeddings of certain imaginary 
quadratic fields into maximal orders of the quaternion algebra $B$, leading to a 
larger amount of supersingular zeros.

\bigskip

\noindent {\bf Acknowledgments:} We would like to thank the ICTP and the ICERM 
for 
provding the oportunity of working on this project. We would like to thank A. 
Pacetti for his comments on an early version of this draft.

\section{Proof of the main theorems} \label{section:proofs}

\subsection{ Katz's modular forms}
We will recall the definition of modular forms given by Katz in \cite{Katz1}. 
\begin{definition}  \label{def:Katzmodularforms}
A modular form of weight $k \in \Z$ and level $1$ over a commutative ring
$R_{0}$ is a rule $g$ that assigns to every pair $(\tilde{E}/R, \omega)$, where 
$\tilde{E}$ is and elliptic
curve over $Spec(R)$ for $R$ an $R_{0}$-algebra and $\omega$ is a nowhere vanishing section of $\Omega^{1}_{\tilde{E}/R}$ on $\tilde{E}$, an element $g(\tilde{E}/R,\omega) \in R$
that satisfies the following properties:

\begin{enumerate}
 \item $g(\tilde{E}/R,\omega)$ depends only on the $R$-isomorphism class of $(\tilde{E}/R,\omega)$.
 \item For any $\lambda \in R^{\times}$, 
 
 \[g(\tilde{E}/R,\lambda \omega)=\lambda^{-k} g(\tilde{E}/R,\omega) .\]
 
 \item $g(\tilde{E}/R,\omega)$ commutes with base change by morphisms of $R_{0}$-algebras.
\end{enumerate}

The space of modular forms of weight $k$ and level $1$ over $R_{0}$ is denoted by $\mathcal{M}(R_{0},k,1)$.

Given any $g \in \mathcal{M}(R_{0},k,1)$, we say that $g$ is holomorphic at 
$\infty$ if its $q$-expansion,  
\[g((Tate(q),\omega_{can})_{R_{0}} \in \Z((q))\otimes_{\Z} R_{0} ,\]
actually belongs to $\Z[[q]]\otimes_{\Z} R_{0}$. The submodule of all such 
elements will be denoted by $M(R_{0},k,1)$.

\end{definition}

\begin{remark}
 The reader should notice that the notations used here are not the same as the 
ones used by Katz.
\end{remark}

In the rest of the article we will only consider the case when $R_{0}=\overline{\F}$, for $p \ge 5$ a prime number.

In \cite{Ser} and \cite{Ser2} Serre considers the space  of modular forms modulo $p$ of weight $k$ and level $1$  as the space consisting of all elements of $\overline{\F}[[q]]$ that are the reduction modulo $p$ of the $q$-expansions  of elements in $M_{k}$ that have $p$-integer coefficients. The following proposition shows that under mild assumptions, this definition agrees with the previous definition.

\begin{proposition}[\cite{Edixhoven2} Lemma $1.9$] Let $k \ge 2$ and $p \ge 5$. 
Then, the natural map
 
 \[ M(\overline{\Z_{p}},k,1) \rightarrow M(\overline{\F},k,1) ,\]
 is surjective.
\end{proposition}

\begin{example}
Given $p \ge 5$, and an elliptic curve $\tilde{E}/\overline{\F}$ we can write an equation for $\tilde{E}$ of the form

\[ \tilde{E}: y^2=x^3-27c_{4}-54c_{6}  .\]

It is equipped with a canonical nowhere vanishing differential $\omega_{can}= \frac{dx}{y}$.

 \begin{itemize}

  \item $E_{4}(\tilde{E}/\overline{\F},\omega_{can}):= c_{4}$ 
defines an element  in $M(\overline{\F},4,1)$ whose $q$-expansion is the same 
as 
the
the reduction modulo $p$ of the classical Eisenstein series $E_{4}$.
   \item $E_{6}(\tilde{E}/\overline{\F},\omega_{can}):= c_{6}$ 
defines an element  in $M(\overline{\F},6,1)$ whose $q$-expansion is the same 
as 
the
the reduction modulo $p$ of the classical Eisenstein series $E_{6}$.
\item  $\Delta(\tilde{E}/\overline{\F},\omega_{can}):= 
\frac{c_{4}^3-c_{6}^2}{1728}=\Delta(\tilde{E})$ defines an element  in 
$M(\overline{\F},12,1)$ whose $q$-expansion is the same as the
the reduction modulo $p$ of the classical cuspform $\Delta$.
\item  $j(\tilde{E}/\overline{\F},\omega_{can}):= 
\frac{c_{4}^3}{\Delta}=j(\tilde{E})$ defines an element  in 
$\mathcal{M}(\overline{\F},0,1)$ whose $q$-expansion is the same as the
the reduction modulo $p$ of the classical $j$-invariant.
 \end{itemize}

\end{example}

\begin{proposition} \label{prop:nonvanishing}
 Given $\tilde{E} / \overline{\F}$ an elliptic curve and $\omega$ a nowhere 
vanishing differential on $\tilde{E}$, the following holds:
 \begin{itemize}
  \item  $\Delta(\tilde{E},\omega)$ never vanishes.
  \item  $ E_{4}(\tilde{E},\omega)$ vanishes if and only if $j(\tilde{E})=0$.
  \item $E_{6}(\tilde{E},\omega)$ vanishes if and only if $j(\tilde{E})=1728$.
  \item $j((\tilde{E},\omega))=j(\tilde{E})$, i.e., it only depends on the isomorphism class of $\tilde{E}$.
  
 \end{itemize}

\end{proposition}

\begin{proof}
 If we evaluate $\Delta(\tilde{E}, \omega_{can})$ we recover the discriminant of 
$\tilde{E}$. This is non-zero as, by definition, an elliptic curve is 
non-singular. The remaining statements are analogous.
\end{proof}

Now we have the ingredients to prove the following proposition that relates the zeros of the divisor polynomial of $E$ with the zeros of the modular form $F_E$
modulo $p$.

\begin{proposition} \label{prop:zerosmodp}
 Given $\tilde{E}/\overline{\F}$ an elliptic curve with a nowhere vanishing invariant differential $\omega$ we have that
 
 \[ F(F_{E},j(\tilde{E})) \equiv 0 \bmod{p}  \iff F_{E}( \tilde{E}, \omega)= 0  .\]
 
 
\end{proposition}

\begin{proof}
Suppose that $j(\tilde{E}) \not= 0,1728$. Consider  
\[ \frac{F_E}{\Delta^{m(k)} 
\tilde{E}_{k}}= \tilde{F}(F_{E},j(-)) \in 
\mathcal{M}(\overline{\F},0,1).  \]
It can be evaluated at pairs $( \tilde{E}, \omega)$, but since
 it is has weight zero it depends only on the isomorphism class of $\tilde{E}$. Therefore it only depends on the $j$-invariant of the elliptic curve. Note that
by Proposition \ref{prop:nonvanishing} the denominator does not vanish and the result follows. If $j=0$ or $j=1728$ an analogous argument shows the proposition, as
$F(F_{E},x)=h_{k}(x)\tilde{F}(F_{E},x)$, and $h_{k}$ takes into account the vanishing of these special $j$-invariants.
 
\end{proof}

\subsection{The spaces $S(\overline{\F},k,1)$ }

Following \cite{Serre2}, we introduce a definition:

\begin{definition} \label{def:supersingular}
 ${S}(\overline{\F},k,1)$ is the space of rules $g$ that assign to every 
pair $(\tilde{E}/\overline{\F}, \omega)$, where $\tilde{E}$ is a 
\textbf{supersingular} elliptic
curve and $\omega$ is a nowhere vanishing differential on $\tilde{E}$, an 
element $g(\tilde{E}/\overline{\F},\omega) \in \overline{\F}$
that satisfies the same properties as in Definition \ref{def:Katzmodularforms}.

\end{definition}

\begin{definition} \label{def:Hecke}
For $\ell \neq p$ a prime number we define the Hecke operator $T_{\ell}$ acting on $S(\overline{\F},k,1)$ as

\[  (g \mid_{T_{\ell}})(\tilde{E},\omega)= \frac{1}{\ell} \sum_{C} g(\tilde{E}/C, {\pi_{C}}_{\ast} \omega)  ,\]

where the sum is taken over the $\ell+1$ subgroups of $\tilde{E}$ of order $\ell$ and  $\pi_{C}: \tilde{E} \rightarrow  \tilde{E}/C$ is the corresponding isogeny.
 
\end{definition}

\begin{proposition} \label{inclusion}
We have a natural inclusion $M(\overline{\F},k,1) \subset {S}(\overline{\F},k,1)$. If $g \in M(\overline{\F},k,1)$ is an eigenform for the Hecke operators $T_{\ell}$ ($\ell \neq p$)
with eigenvalues $a_{\ell} \in \overline{\F}$, then, the image of $g$ in $ {S}(\overline{\F},k,1)$ is an eigenform for the Hecke operators with the same
eigenvalues $a_{\ell}$.
 \end{proposition}

\begin{proof}
 This is clear from the definitions.
\end{proof}

We have the following proposition  that allows us to 
shift from weight $p+1$ to weight $0$.

\begin{proposition}[\cite{Robert}, Lemma ~6] \label{prop:Robert}
 The map from ${S}(\overline{\F},0,1) \rightarrow {S}(\overline{\F},p+1,1)$ given by multiplication by $E_{p+1}$ induces an isomorphism of Hecke modules
 
\[{S}(\overline{\F},0,1)[1] \cong {S}(\overline{\F},p+1,1),  \]

where  ${S}(\overline{\F},0,1)[1]$ denotes the Tate twist. More precisely we have that for all $g \in {S}(\overline{\F},0,1)$,

\[ \ell E_{p+1}.  (g\mid_{T_{\ell}}) = (g.E_{p+1}) \mid_{T_\ell} .\]
\end{proposition}

If we consider the isobaric polynomials $A,B$ such that $A(E_4,E_6)=E_{p-1}$ 
and $B(E_4,E_6)=E_{p+1}$, the reductions $\tilde{A}$, $\tilde{B}$ have no 
common factor (\cite[Corollary ~1 of Theorem ~5]{Ser}).  Since $E_{p-1}$ 
vanishes  at supersingular elliptic curves we obtain that 
$E_{p+1}$ does not vanish at supersingular elliptic curves over 
$\overline{\F}$.

The reduction modulo $p$ of $F_{E}$ can be regarded as an element of 
${S}(\overline{\F},p+1,1)$, and by the above remarks we can consider

\[ \overline{F_{E}}= F_{E}/E_{p+1} .\]

Combining these results with Proposition \ref{prop:zerosmodp}  we obtain the 
following result.

\begin{proposition} \label{prop:zerosmodp2}
Given $\tilde{E}/\overline{\F}$ a \textbf{supersingular}  elliptic curve with a nowhere vanishing invariant differential $\omega$ we have that
 
 \[ F(F_{E},j(\tilde{E})) \equiv 0 \bmod{p}  \iff \overline{F_{E}}( \tilde{E})= 0 .\]
\end{proposition}

Finally, we state a proposition that will be useful later.

\begin{proposition} \label{prop:eigmodp}
 The element $\overline{F_{E}} \in {S}(\overline{\F},0,1)[1]$ has the same 
eigenvalues for $T_{\ell}$ ($\ell \neq p$) as $F_{E}$. In addition,
 it has the same eigenvalues modulo $p$ as $f_{E}$.
\end{proposition}

\begin{proof}
The first part follows from  Proposition \ref{prop:Robert} while the second part 
follows from the discussion given in the introduction.

\end{proof}
\subsection{Modular forms on quaternion algebras}
We will recall some of the results previously stated in the introduction. This exposition follows entirely the fundamental work of Gross \cite{Gross}.
The geometric fiber of the curve $X_0(p)$ in characteristic $p$ is the union of two rational curves meeting at $n$ ordinary double points: $e_1$, $e_2$, \ldots, $e_{n}$ that are in bijective correspondence with the isomorphism classes of supersingular elliptic curves $E_i$.
 Recall that $\X$ is the free $\Z$-module of divisors supported on the $e_i$ with a $\Z$-bilinear pairing 
\[
\langle,\rangle:\X\times\X \rightarrow \Z,
\]
given by $\langle e_i,e_j \rangle=w_i \delta_{i,j}$ for all $i,j \in \{1,\ldots,n \}$,
where $w_{i}=\frac{1}{2}\#\textrm{Aut}(E_i)$.

This pairing identifies ${\X}^{*}=Hom(\X,\Z)$ with the subgroup of $\X \otimes \Q$ with basis $e^{*}_{i}=\frac{e_{i}}{w_{i}}$.

The action of Hecke correspondences on the set of $e_i$ induces an action on $\X$. Explicitly, the action of the correspondence $t_m$ ($m\ge 1$) is given by the transpose  of the Brandt matrix $B(m)$
\[
t_m e_i = \sum_{j=1}^n B_{ij}(m)e_j,
\]

where $B_{ij}(m)$ is the number of subgroups schemes of order $m$ in $E_{i}$ such that $E_{i}/C \simeq E_j$.
Furthermore, the pairing is Hecke compatible \cite[Proposition ~4.6]{Gross}.

Let $M_2$ be the $\Z$-module consisting of holomorphic modular forms for the 
group $\Gamma_{0}(p)$ such that when we consider its $q$-expansion, all 
coefficients are integers except maybe the coefficient $a_0$ which is only 
required to be in $\Z[1/2]$.
 The Hecke algebra $\mathbb{T}=\Z[ \cdots , T_m,  \cdots ]$ acts on $M_2$  by the classical formulas. Moreover, we have that as endomorphisms of
$M_2$ 
\[T_{p}+W_{p}=0, \] where $W_p$ is the Atkin-Lehner involution. In addition, 
the map given by $T_m \rightarrow  t_m$ defines an isomorphism of Hecke 
algebras.

\begin{proposition}[\cite{Gross}, Proposition ~5.6]\label{prop:Gross}
The map $\phi: \X \otimes_{\mathbb{T}} \X \rightarrow M_2$ given by

 \[ \phi(e,f)=\frac{deg(e)deg(f)}{2} + \sum_{m \ge 1}  \langle t_{m}e,f \rangle q^m, \]
  defines a $\mathbb{T}$-morphism which becomes an isomorphism over $\mathbb{T} \otimes \Q$.
\end{proposition}

Now we can define 
\[v_E=\sum v_E(e_i)e_i\in \X \] to be an eigenvector for all $t_m$ corresponding to $f_E$, i.e. $t_m v_E= a(m) v_E$, where $f_E(\tau)=\sum_{m=1}^\infty a(m)q^m$. We normalize $v_E$ (up to the sign) such that the greatest common divisor of all its entries is $1$. The key observation is that $v_E$ has the same eigenvalues modulo $p$ as $F_{E}$.

The rule \[ \overline{F_{E}}= F_{E}/E_{p+1}  \in {S}(\overline{\F},0,1) \] can be evaluated at supersingular elliptic curves over
 $\overline{\F}$ (it has weight zero), and by duality, it defines an element $\overline{F_{E}}^{*} \in \overline{\mathcal{X}}$, where $\overline{\X}$ is the reduction modulo $p$ of $\X$. 
 
 \begin{proposition} \label{prop:sameeigenvalues}
 $\overline{F_{E}}^{*}=\sum  \overline{F_{E}}( e_i) e^{*}_{i}= \sum \overline{F_{E}}( e_i)(1/w_{i}) e_{i}$ and $v_{E}=\sum  {v_{E}}( e_i) e_{i}$ have the same eigenvalues modulo $p$ for the Hecke operators $T_{\ell}$ ($\ell \neq p$).
 
 \end{proposition}

\begin{proof}
 By Proposition \ref{prop:eigmodp}, $\overline{F_{E}}$ has the same eigenvalues as $F_{E}$ for $T_\ell$ ($\ell \neq p$), but with the action twisted.  
Note that $t_{\ell}$ and the action of $T_{\ell}$ on $S(\overline{\F},0,1)$ differ by precisely this factor $\ell$, therefore the result follows
since $v_{E}$ has the same eigenvalues modulo $p$ as $F_E$ and the pairing is Hecke-linear.
\end{proof}

%

\begin{corollary} \label{coro:last}
$\overline{F_{E}}( e_i) \equiv 0 \bmod{p}  \iff  {v_{E}}(e_i) \equiv 0  \bmod{p}$.
\end{corollary}

\begin{proof}
The forms $\overline{F_{E}}^{*}=\sum \overline{F_{E}}( e_i)(1/w_{i}) e_{i}$ and 
$v_{E}=\sum  {v_{E}}( e_i) e_{i}$ have the same eigenvalues
for $T_{\ell}$  ($\ell \neq p$) by Proposition \ref{prop:sameeigenvalues}. By 
the work of Emerton \cite[Theorem ~0.5 and Theorem ~1.14]{Emerton} we 
have the multiplicity one property for $\X$ modulo $p$, since $p$ is a prime 
different from $2$. 

Therefore, up to a non-zero scaling, the coefficients of these two quaternion 
modular forms agree modulo $p$. Finally, noting that the $w_{i}$ are not 
divisible by $p$,
the result follows.
\end{proof}

Now we are in position to prove Theorem \ref{thm:zeros}.
\begin{proof}[Proof of Theorem \ref{thm:zeros}]
 \[ {v_{E}}( e_i) \equiv 0 \iff  \overline{F_{E}}( e_i)=0  \iff  F(F_{E},j(E_{i})) \equiv 0 .   \]
 The first equivalence is Corollary \ref{coro:last}; the last one is 
Proposition \ref{prop:zerosmodp2}.
\end{proof}

Let $S_p \subset \{1,\ldots, n\}$ be a subset of indices such that $i \in S_p$ if and only if $j(E_i)\in \F$ (hence $\#S_p=s_p$). For $i \in \{1,\ldots, n\}$ let $\bar{i}$ be the unique element of $\{1,\ldots, n\}$  such that $E_i^p \cong E_{\bar{i}}$. Note that, $\bar{i}=i$ if and only if $i \in S_p$.

\begin{proposition}[\cite{Gross}, Proposition ~2.4] \label{prop:tp}
The Hecke operator $t_{p}$ induces an involution on $\X$ which satisfies that for every $1 \le i \le n$

\[t_p e_{i}=e_{\bar{i}}  .\]

\end{proposition}

Now we finish the section with the proof of Theorem \ref{thm:1}.

\begin{proof}[Proof of Theorem \ref{thm:1}]
Let $E_{i}$ be a supersingular elliptic curve with $j(E_{i}) \in \F$. The 
operator $t_{p}$ acts as $-W_{p}$ on $M_2$ and since the elliptic curve has root 
number $-1$ we get that $t_{p}$ acts as $-1$. By Proposition \ref{prop:tp} we 
have that $t_{p}e_i=e_i$, hence ${v_{E}}( e_i)=0$, and the result follows from 
Theorem \ref{thm:zeros}.
\end{proof}
\section{Proof of Theorem \ref{thm:even}} \label{section:mod2}

\subsection{Some basic properties of Brandt matrices}
Following \cite{Gross}, we will recall some useful properties of Brandt matrices. 
Let $B$ be the quaternion algebra over $\Q$ ramified at $p$ and $\infty$. For 
each $i=1,\ldots,n$ let $R_i$ be a maximal order of $B$ such that $R_i \cong 
End(E_i)$. Set $R=R_1$ and  let $\{I_1, \ldots, I_n\}$ be a set of left 
$R$-ideals representing different $R$-ideal classes, with $I_1=R$. We can choose 
the
$I_i$'s such that the right order of $I_i$ is equal to $R_i$. For $1\le i,j \le 
n$, define $M_{ij}=I_j^{-1}I_i$; this is a left $R_i$-module and a right 
$R_j$-module. The Brandt matrix of degree $m$, $B(m)=(B_{ij}(m))_{1\le i,j \le 
n}$, is 
defined by the formula
\[
B_{i,j}(m)=\frac{1}{2 w_j}\# \{ b \in M_{ij}: \frac{\Nr(b)}{\Nr(M_{ij})} = m\},
\]
where $\Nr(b)$ is the reduced norm of $b$, and $\Nr(M_{ij})$ is the unique positive rational number such that the quotients $\frac{\Nr(b)}{\Nr(M_{ij})}$ are all integers with no common factor. 

Alternatively, $M_{ij} \cong Hom_{\overline\F}{(E_i, E_j)}$ and $B_{i,j}(m)$ is equal to the number of subgroup schemes $C$ of order $m$ in $E_i$ such that $E_i/C \simeq E_j$ \cite[Proposition ~2.3]{Gross}. 

Following the discussion before \ref{prop:tp} we can state the following results.

\begin{proposition} \label{prop:sameparity}
We have the equality $v_E(e_j)=\lambda_{p} v_E(e_{\bar{j}})$. In particular, $v_E(e_j)$ and $v_E(e_{\bar{j}})$ have the same parity.
\end{proposition}

\begin{proof}
The first assertion follows from the fact that $\sum_{i} v_E(e_i)e_{i}$ is an eigenvector for the action of $t_p$ and Proposition \ref{prop:tp}. The last assertion follows from the fact that $\lambda_p= \pm 1$.
\end{proof}

\begin{proposition}\label{prop:Bij}

For all $i,j \in \{1, \ldots, n \}$ and $m\in \N$, we have
$$B_{ij}(m)=B_{\bar{i}\bar{j}}(m).$$ 

\end{proposition}

\begin{proof} 
For any $m$ we have that, since the Brandt matrices commute, $B(m)B(p)=B(p)B(m)$. In other words, 
\[\sum_k B_{\bar{i}k}(p) B_{kj}(m)= \sum_k B_{\bar{i}k}(m) B_{kj}(p).\]

Using Proposition \ref{prop:tp} we know that $B_{k \ell}(p)=\delta_{ \bar{k} \ell}$, in consequence we have

\[ B_{ij}(m)= B_{\bar{i} \bar{j}}(m)  ,\]
as we wanted. 
\end{proof}

\begin{proposition}\label{prop:even} Let $l\ne p$ be an odd prime such that $\left(\frac{-p}{l}\right)= -1$.
Then for all $i,j \in S_p$,
\[
B_{ij}(l) \equiv 0 \pmod{2}.
\]
\end{proposition}
\begin{proof}
Let $\phi_i\in R_i \cong End(E_i)$ and $\phi_j \in R_j\cong End(E_j)$ be the Frobenius endomorphisms of the elliptic curves $E_i$ and $E_j$ respectively  (they exist since $E_i \cong E_i^{p}$ and $E_j \cong E_j^{p}$). These are trace zero elements of reduced norm $p$, i.e. $\phi_i^2=\phi_j^2=-p$. Consider the map $\Theta: B \rightarrow B$ given by
\[
\Theta(f) = \frac{-1}{p}\phi_j f \phi_i.
\]
Note that $\Theta^2=Id$, and $\Nr(\Theta(f))=\Nr(f)$.

First we prove that $\Theta(M_{ij})\subset M_{ij}$. Take $f\in Hom(E_i, E_j)$ and consider 
\[ g=\phi_j\circ f \circ \phi_i \in Hom(E_i, E_j).\] Since the inseparable degree of $g$ is divisible by $p^2$, it factors as $h\circ [p]$ with $h\in Hom(E_i, E_j)$,
hence $\Theta(f)$ belongs to $Hom(E_i, E_j)$.

Next, we show that $\Theta$ has two eigenspaces $W_-$ and $W_+$ of dimension $2$ with eigenvalues $-1$ and $1$ respectively. We consider two cases:

\begin{itemize}
\item [a)] $i=j$ (i.e. $M_{ij}=R_i$)
	
\noindent Direct calculation shows that the vectors $1$ and $\phi_i$ span the 
eigenspace with eigenvalue $1$. The eigenspace with eigenvalue $-1$ is the 
orthogonal complement of $\phi_i$ in the trace zero subspace $B^{0}$ of $B$ 
(since for $f \in B^{0}$ we have $f \perp \phi_i \iff 
\Nr(f+\phi_i)=\Nr(f)+\Nr(\phi_i) \iff f \hat{\phi_i}+\hat{f} \phi_i=0 \iff 
f\phi_i = -\phi_i f \iff \Theta(f)=-f$).
\item[b)] $i \ne j$
	
\noindent Let $\phi_{ji}:= \phi_j \phi_i$. The matrix representations of $\Theta$ in the invariant subspaces generated by $\{ 1, \phi_{ji}\}$ and $\{ \phi_i, \phi_j \}$ are equal to $\sm 0 {-p} {-1/p} 0$ and $\sm 0 {1} {1} 0$, hence  $\Theta$ has two eigenspaces of dimension $2$ with eigenvalues $-1$ and $1$.
\end{itemize}

For $b \in M_{ij}$ let $w_1\in W_-$ and $w_2 \in W_+$ be such that $b=w_1+w_2$. Then $\Theta(b)=-w_1+w_2\in M_{ij}$, and $2w_1, 2w_2 \in M_{ij}$. Let $V_-=W_- \cap M_{ij}$	and $V_+= W_+ \cap M_{ij}$. Thus 
\[ M_{ij}/(V_- + V_+) \le \Z/2\Z + \Z/2\Z. \] 

In order to prove that $B_{ij}(l)$ is even, it is enough to show that for every $b\in M_{ij}$ such that $\frac{\Nr(b)}{\Nr(M_{ij})} = l$ the set 
\[ C=\{\omega b: \omega \in R_j^\times\}\cup \{\omega \Theta(b): \omega \in R_j^\times \} \] has maximal cardinality $\#C=4w_j$ (note that all elements of $C$ have the same norm.) It is enough to prove that $b$ is not an eigenvector of $\Theta$.

Let $a\in \Z$ be such that $I=aM_{ij}\subset R_j$. If $M^2$ is the index of $I$ in $R_j$, then $q_I(x):=\frac{\Nr(x)}{M}$ is an integral quadratic form on $I$ which is in the same genus as $(R_j,\Nr)$. In particular, $disc(q_I)=p^2$. Moreover, $q(x):=q_I(ax)$ is a quadratic form on $M_{ij}$ for which $q(x)=\frac{\Nr(x)}{\Nr(M_{ij})}$ ($\Nr(M_{ij})=\frac{1}{M}$).
Since $\Theta$ preserves reduced norm, the lattices $V_+$ and $V_-$ are orthogonal with respect to $q$, and $|disc(V_+)disc(V_-)|=|disc(V_+ + V_-)|$. It follows that 
\[disc(V_+), disc(V_-) \in \{-p,-4p \} \] since $M_{ij}/(V_- + V_+) \le \Z/2\Z + \Z/2\Z$ and $q$ is a positive definite form.

Assume that $b$ is an eigenvector of $\Theta$. Then $b \in V_+$ or $b\in V_-$. In any case since $l=q(b)$, it follows that $l$ is representable by a binary quadratic form of discriminant $-p$ or $-4p$ which is not possible since 
$\left(\frac{-p}{l}\right)=\left(\frac{-4p}{l}\right)= -1$.

\end{proof}

\subsection{Fourier coefficients of $f_E(\tau)$ mod $2$}
\begin{proposition}\label{prop:ell} Let $E/\Q$ be an elliptic curve of prime conductor $p$ such that $E$ has positive discriminant and $E$ has no rational point of order $2$. There is a positive proportion of odd primes $\ell$ such that $\left(\frac{-p}{\ell}\right)= -1$ and $a(\ell)\equiv 1 \pmod{2}$, where $f_E(\tau)=\sum a(n)q^n$ is the $q$-expansion of $f_E(\tau)$. 
\end{proposition}
\begin{proof}
Denote by $\rho_2:\Gal(\overline{\Q}/\Q) \rightarrow \GL_2(\mathbb{F}_2)$ the 
mod $2$ Galois representation attached to the elliptic curve $E$ (or 
equivalently, by the modularity theorem, to the modular form $f_E$). For an odd 
prime $\ell\ne p$, we have that 
\[ a(\ell) \equiv \Tr(\rho_2(Frob_{\ell}))\bmod{2}, \] 
where $Frob_\ell$ is a Frobenius element over $\ell$. The group 
$\GL_2(\mathbb{F}_2)$ is isomorphic to $S_3$, and the elements of trace $1$ are 
exactly the elements of order $3$. $\rho_2$ factors through $\Gal(K/\Q)$, and 
$\Gal(K/\Q) \cong \Im(\rho_2)$ where $K=\Q(E[2])$. It is enough to prove that 
there is a positive proportion of prime numbers $\ell$ such that  
$\left(\frac{-p}{\ell}\right)= -1$ and $Frob_\ell\in \Gal(K/\Q)$ has order $3$. 
Since $E$ has no rational point of order $2$, $Gal(K/\Q)$ is either $\Z/3\Z$ 
(if the discriminant of $E$ is a square) or $S_3$. 
Moreover, since $E$ has prime conductor and no rational two torsion, it follows 
from Proposition $7$ in  \cite{Serre} that the absolute value 
of the discriminant 
is not a square. Hence, $K/\Q$ is an $S_3$ extension, and since the discriminant 
is positive and its only prime divisor can be $p$, the quadratic field $F$ 
contained in $K$ is equal to $\Q(\sqrt{p})$.

If $\ell \equiv 3 \pmod{4}$ then  $\left(\frac{-p}{\ell}\right)= -1$ implies 
that $\ell$ splits in $F$. If, in addition, $\ell$ does not split completely in 
$K$, then the order of $Frob_\ell$ is $3$ and $a(\ell)$ is odd. There is a 
positive proportion of such primes $\ell$ since by Chebotarev density theorem 
(applied to the field $L=\Q(\sqrt{-1})K$) there is a positive 
proportion of primes $\ell$ which are inert in $\Q(\sqrt{-1})$, split in 
$F$ and do not split completely in $K$. 

\end{proof}

\subsection{Proof of Theorem \ref{thm:even}}

\begin{proof}
Take $\ell$ an odd prime such that $\left(\frac{-p}{\ell}\right)= -1$ and 
$a(\ell)\equiv 1 \pmod{2}$ as in Proposition \ref{prop:ell}. Consider the action 
of $t_{\ell}$ on $\sum_i v_E(e_i)e_i$. Take any $j \in S_p$, that is 
$\bar{j}=j$. By comparing the coefficient of $e_j$ in the equation 
$t_{\ell}\sum_i v_E(e_i)e_i= \lambda_{\ell} (\sum_i v_E(e_i)e_i)$ we obtain

\[ \lambda_{\ell}v_E(e_{j})= \sum_i v_E(e_i)B_{ij}(\ell)  .\]
 We are going to look at this equation modulo $2$; we know that $\lambda_{\ell}=\ell+1-a_{\ell}$ is odd and we know by Proposition \ref{prop:even} that for any $i \in S_{p}$, $B_{ij}(\ell)$ is even. Therefore,
 
 \[  v_E(e_j) \equiv \sum_{i \not \in S_p} v_E(e_i) B_{ij}(\ell) \bmod{2}  . \]
 
 Proposition \ref{prop:Bij} tells us that 
$B_{ij}(\ell)=B_{\bar{i}\bar{j}}(\ell)=B_{\bar{i}j}(\ell)$ as $j=\bar{j}$. 
Moreover, by Proposition \ref{prop:sameparity}, the numbers $v_E(e_{i})$ and 
$v_E(e_{\bar{i}})$ have the same parity. Therefore, rearranging the elements of 
the sum $\sum_{i \not \in S_p} v_E(e_i) B_{ij}(\ell)$ in conjugated pairs, we 
obtain that this sum is zero modulo $2$. In conclusion we must have 
$v_E(e_j) \equiv 0 \bmod{2}$, as we wanted to prove.
\end{proof}

\bigskip
We are going to give a different proof of Theorem \ref{thm:even} under the 
additional assumption that $E$ is supersingular at $2$. 
The idea is to use the results of \cite{LeH} on level raising modulo $2$ 
together with the multiplicity one mod $2$ results from \cite{Emerton} to obtain 
mod $2$ congruences between modular forms of the same level $p$, but with 
different signs of the Atkin-Lehner involution. We hope that by extending these 
ideas to level $2^r p$ one will be able to understand Conjecture \ref{conj:1} 
better. 

\begin{theorem}
Let $E$ be a rational elliptic curve of conductor $p$, without rational $2$-torsion and with positive discriminant. Suppose further that $E$ is supersingular
at $2$ . Then, there exists a newform $g \in S_{2}(\Gamma_{0}(p))$ and a prime $\lambda$ above two in the field
of coefficients of $g$ such that $f \equiv g \bmod{\lambda}$ and such that $W_{p}$ acts as $-1$ on $g$.
\end{theorem}

\begin{proof}
We will verify the assumptions of \cite[Theorem ~2.9]{LeH}, starting with our elliptic curve $E$ of prime conductor and in the scenario where we choose
no primes as level raising primes (so we are looking for a congruence between level $p$ newforms).
As we explained before, the hypotheses imply that ${\rho}_{2}: G_{\mathbb{Q}} 
\rightarrow Gl_{2}(\mathbb{F}_{2})$ is surjective and the only quadratic 
extension of $\mathbb{Q}(E[2])$ is given by $\mathbb{Q}(\sqrt{p})$. Therefore, 
the conductor of ${\rho}_{2}$ is $p$  and it is not induced
from $\mathbb{Q}(i)$. Moreover ${\rho}_{2}$ restricted to $G_{\mathbb{Q}_{2}}$ 
is not trivial if  $E$ is supersingular at $2$. Thus, we are in position to 
use the theorem and find a $g$ as in the statement, because, since $\Delta(E)>0$, we can prescribe the sign of the Atkin-Lehner involution at $p$.
\end{proof}

Now we are in condition to give another proof of Theorem \ref{thm:even}, under the additional assumption that $E$ is supersingular at $2$. 
Since $g$ has eigenvalue $-1$ for the Atkin-Lehner operator we have that $v_g(e_i)=0$
for every $i \in S_{p}$ by Proposition \ref{prop:tp}. As we did earlier, Theorem 
~0.5 and Theorem ~1.14 in \cite{Emerton} imply, since $E$ is supersingular 
at $2$, that we have multiplicity one mod $2$ in the $f_E$-isotypical component 
in $\X$, therefore $v_E(e_{i})$ is even for $i \in S_{p}$ as we wanted to show.

\section{Further remarks}
Suppose that $E$ is an elliptic curve with root number $+1$ and positive rank. 
By Gross-Zagier-Kolyvagin we must have $L(E,1)=0$ and we can use 
Gross-Waldspurger formula to obtain some relations satisfied by the 
$v_{E}(e_{i})$. More precisely if we take $-D$ a fundamental negative 
discriminant define
\[ b_D= \sum_{i=1}^{n} \frac{h_i(-D)}{u(-D)} e_i  ,\]
where $h_i(-D)$ is the number of optimal embeddings of the order of discriminant 
$-D$ into $End(E_{i})$ modulo conjugation by $End(E_{i})^{\times}$ and
$u(-D)$ is the number of units of the order.
In this scenario, we have Gross-Waldspurger formula (\cite[Proposition ~13.5]{Gross}).

 \begin{proposition} \label{prop:Gross}
 If $-D$ is a fundamental negative discriminant with $ 
\left(\frac{-D}{p}\right)=-1$, then
 
 \[ L(E,1)L(E \otimes \varepsilon_D,1)= \frac{\left(f_E,f_E \right)}{\sqrt{D}} 
\frac{{m_D}^2}{ \langle v_E, v_E \rangle } ,\] 
 
 where $\varepsilon_D$ is the quadratic character associated to $-D$, 
$\left(f_E,f_E \right)$ is the Petersson inner product on $\Gamma_{0}(p)$ and
$m_D= \langle v_E, b_D \rangle$.

\end{proposition}

Since  $L(E,1)=0$ we obtain that 
\[m_D=\langle v_E, b_D 
\rangle=0.\] This says that, as we vary throughout all $D$ as in the 
proposition, we obtain some relations that are satisfied by the $v_{E}(e_{i})$ 
that make them more likely to be zero. For example, if we take a fundamental 
discriminant of class number $1$ such that $p$ is inert in that field, then the 
divisor $b_{D}$ is supported in only one $e_{i}$ with $i \in S_{p}$. Since the 
inner product between $b_{D}$ and $v_{E}$ is zero we get that $v_{E}(e_{i})=0$. 
This certainly explains a lot of the vanishing that is occurring in our 
setting, specially considering that the range we are looking into is not very 
large. One could hope to make these heuristics more precise by analyzing  
imaginary quadratic fields with small size compared to the degree of the 
modular parametrization (this measures the norm of $v_E$) and try to 
obtain explicit lower bounds on the number of zeros in this situation.

\end{document}